\documentclass{article}
\usepackage{graphicx} 
\usepackage{amsmath} 
\usepackage{amssymb}
\usepackage{amsthm}
\usepackage{color}
\usepackage{url}

\newcommand{\Par}[1]{\left({#1}\right)}
\newcommand{\Norm}[1]{\left\|{#1}\right\|}
\newcommand{\Prod}[2]{\Par{#1, #2}}

\newtheorem{lemma}{Lemma}
\newtheorem{theorem}{Theorem}

\newcommand{\revision}[1]{\textcolor{black}{#1}}

\newcommand{\inred}[1]{\textcolor{red}{#1}}

\title{Stabilization of isogeometric finite element method with optimal test functions computed from $L_2$ norm residual minimization}
\author{Marcin \L{}o\'s$^1$, \revision{Tomasz S\l{}u\.zalec$^1$}, Maciej Paszy\'nski$^1$, Eirik Valseth$^{2,3,4}$}
\date{$^1$Faculty of Computer Science, AGH University, Krak\'ow, Poland \\ $^2$The Norwegian University of Life Science, Ås, Norway \\ $^3$Simula Research Laboratory, Oslo, Norway   \\ $^4$Oden Institute for Computational Engineering and Sciences, The University of Texas at Austin, USA  }

\begin{document}

\maketitle

\begin{abstract}
We compare several stabilization methods in the context of isogeometric analysis
and B-spline basis functions, 
using an advection-dominated advection\revision{-}diffusion as a model problem.
We derive (1) the least-squares finite element method formulation using the framework of Petrov-Galerkin method
with optimal test functions in the~$L_2$ norm, which guarantee automatic preservation
of the \emph{inf-sup} condition of the continuous formulation.
We also combine it with the standard Galerkin method 
to recover (2) the Galerkin/least-squares formulation,
and derive coercivity constant bounds valid for B-spline basis functions.
The resulting stabilization method are compared with the least-squares and 
(3) the Streamline-Upwind Petrov-Galerkin (SUPG)method using again the Eriksson-Johnson model problem.
The results indicate that least-squares
(equivalent to Petrov-\revision{Galerkin} with $L_2$-optimal test functions)
outperforms the other stabilization methods for small P\'eclet numbers,
while strongly advection-dominated problems are better handled with
SUPG or Galerkin/least-squares.
\end{abstract}

\section{Introduction}

In the finite element method \cite{IGA}, the Bubnov-Galerkin method is often employed to numerically solve the weak formulation of partial differential equations (PDEs). 
For certain challenging problems, the Bubnov-Galerkin method  results in a solution which exhibits nonphysical oscillations. A model example of \revision{such behaviour} is the advection-dominated \revision{advection-}diffusion PDE \cite{Chen,Paszynski}, and a particular case of this PDE called the Eriksson-Johnson problem \cite{Eriksson,Calo}.
To ensure physically meaningful numerical solutions that are also convergent, highly specialized finite element (FE) meshes or numerical stabilization methods are required. Some canonical examples of numerical stabilization for FE methods for these  problems are the Streamline-Upwind Petrov-Galerkin method (SUPG) \cite{SUPG},  residual-minimization methods \cite{MinRes}, least-squares FE methods~\cite{lsfem}, and discontinuous Petrov-Galerkin (DPG) methods~\cite{leszek}. The latter three methods are all closely related and differ mainly in the choice of norm in which the FE residual is evaluated. 
Furthermore, these three methods can also be interpreted as special cases of mixed FE methods, see e.g., \cite{Qiu}, since the saddle-point formulation of the residual minimization becomes a mixed problem.
It is also well-known that the residual minimization method is equivalent to the Petrov-Galerkin (PG) formulation with optimal test functions \cite{Paszynski}.


In this paper, we focus on the stabilization formulations in the context of the isogeometric analysis
and B-spline basis functions on regular grids.
We derive (1) the least-squares finite element method as a variant of the residual minimization / DPG method
taking $L_2$ as the norm in which the residual is being minimized.
Then we combine it with the standard Bubnov-Galerkin test functions to create a stabilized
formulation closely related to (2) the Galerkin/least-squares method~\cite{galerkin-ls},
and investigate its stability properties when B-spline basis functions are used for discretization.
We compare these two stabilized formulations (least\revision{-}squares finite element and Galerkin/least-squares)
numerically with (3) the well-known Streamline-Upwind Petrov-Galerkin (SUPG) method. 
The idea is similar to the work of \cite{NURBS}, where Non-Uniform Rational B-Splines are used to develop least-squares FE computations. Our modified stabilized formulation differs from this least\revision{-}squares approach. This allows us to improve the accuracy of the stabilization and numerical results. 

The structure of the paper is the following.
In \revision{S}ection 2.1 we start from a Bubnov-Galerkin formulation of the \revision{two model advection-diffusion} problems.
In Section 2.2 we
apply the residual minimization method with respect to~$L_2$ inner product, in which the optimal test functions can be solved analytically and results in a least-squares finite element method.
In Sections 2.2.1 and 2.2.2
\revision{we show}
that solving the \revision{advection-diffusion model problems} accurately with a standard 
Bubnov-Galerkin method requires manual refinement of the mesh to resolve the boundary layer. This approach requires \emph{a priori} knowledge of the solution\revision{, which} may not be feasible in the general case.
On the other hand, in Section 2.2.3 we show that the least-squares FEM can handle it using 
a relatively coarse, uniform mesh.
In Section 2.3, we combine the standard Bubnov-Galerkin test functions
with~$L_2$ optimal test functions,
and derive a stabilized discrete formulation resembling Galerkin/least-squares.
In the following Section 2.4, we investigate the performance of the scheme numerically
 and compare it with the SUPG method.
 Finally, Section 3 derives the stability results for the mixed Galerkin/ least\revision{-}squares formulations.
The observations and conclusions from our experiments
as well as some future research directions are summarized in Section 4.

\section{Introduction of the methods}

To commence our discussion, we will introduce  three computational methods for the \revision{two benchmarks using the advection-diffusion equations, numerically unstable due to the fact that $\|\beta \| \gg \epsilon $.}

\subsection{\revision{Model problems}}

\revision{Given the unit square domain $\Omega=(0,1)^2$, a convection vector $\beta=(\beta_x,\beta_y)$, we seek the solution of the advection-diffusion equation
\begin{equation}
\label{eq:advection}
\beta_x\frac{\partial u}{\partial x}+\beta_y\frac{\partial u}{\partial x}-\epsilon \left(\frac{\partial^2 u}{\partial x^2}+
\frac{\partial^2 u}{\partial y^2}\right)=0, 
\end{equation}
along with the Dirichlet boundary conditions
$$
u(x,y)=g(x,y) \textrm{ for }(x,y) \in \partial \Omega,
$$
where 
$$
\partial \Omega = \{ (x,y): x \in \{0,1\} y \in (0,1)\} \cup \{(x,y): y \in \{ 0,1 \}, x\in (0,1)\} 
$$ 
is the boundary of $\Omega$.}

\subsubsection{\revision{First model problem}}

\revision{We focus on a model advection–diffusion problem defined as follows. For a unitary square domain $\Omega=(0,1)^2$,  the convection vector $\beta=(1,1)$,
the diffusion coefficient $\epsilon = \{ 0.1, 0.01, 0.003\}$, we seek the solution of the advection-diffusion equation (\ref{eq:advection}), with zero Dirichlet boundary conditions $g(x,y)=0$. This problem develops boundary layers at $x=1$ and $y=1$, and the sharp maximum at point $(1,1)$. This is a very difficult computational problem, for advection $\|\beta\|$ dominating the diffusion  $\epsilon$, it is very difficult to recover the correct shape of the solution at point $(1,1)$. 
The exact solution for this problem is given by
\begin{equation}
u_{exact}(x_1,x_2)=
(x_1 + \frac{\exp(\frac{x_1}{\epsilon})-1}{1-\exp(\frac{1}{\epsilon})})(x_2 + \frac{\exp(\frac{x_2}{\epsilon}-1)}{1-\exp(\frac{1}{\epsilon})})
\label{eq:first}
\end{equation}
The exact solution for $\epsilon \in \{0.1, 0.01, 0.003 \}$ are presented in Figure \ref{fig:first}. For smaller values of $\epsilon$, the numerical evaluation of (\ref{eq:first}) breaks.}

\begin{figure}
    \centering
    \includegraphics[width=0.9\linewidth]{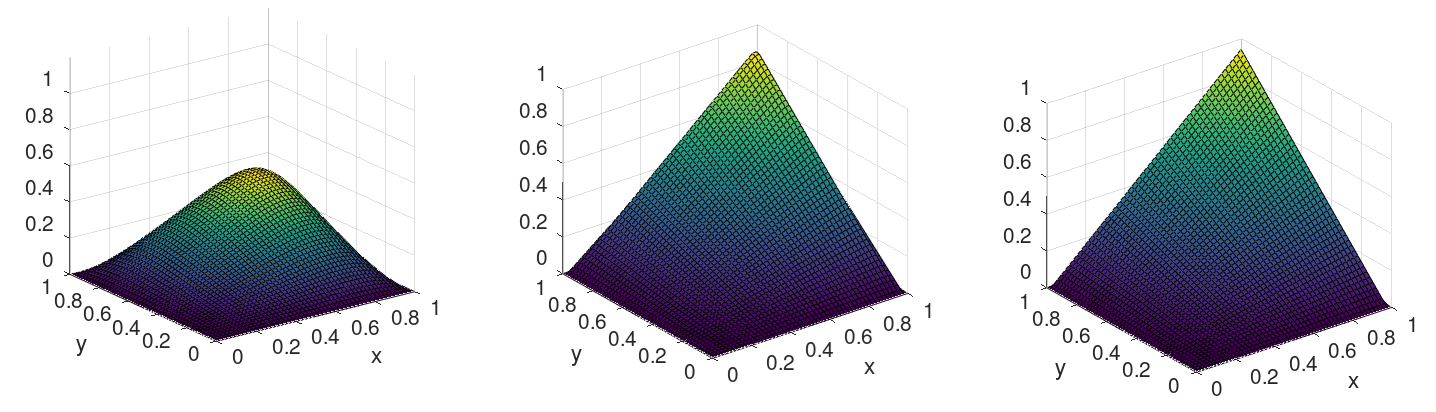}
    \caption{\revision{The exact solutions for the first problem for $\epsilon \in \{0.1, 0.01, 0.003\}$}}
    \label{fig:first}
\end{figure}

\subsubsection{\revision{Second model problem}}

\revision{We consider the famous version of the advection-diffusion problem, the Eriksson-Johnson model problem~\cite{Eriksson}.
On the unit square domain $\Omega=(0,1)^2$, we define a convection vector $\beta=(\beta_x,\beta_y)=(1,0)$, 
along with  Dirichlet boundary conditions:
$$
g(x,y)=0 \textrm{ for }x\in(0,1),y\in\{0,1\} \qquad
g(x,y)=sin(\pi y) \textrm{ for }x=0
$$
The problem is driven by the inflow Dirichlet boundary condition and develops a boundary layer of width $\epsilon$ at the outflow $x = 1$.
The exact solution is given by
\begin{equation}
u(x_1,x_2)=\frac{\exp(r_1(x_1-1))-\exp(r_2(x_1-1))}{\exp(-r_1)-\exp(-r_2)}\sin(\pi x_2)
\end{equation}
where $r_1=\frac{1+\sqrt{1+4\epsilon^2\pi^2}}{2\epsilon}$, $r_2=\frac{1-\sqrt{1+4\epsilon^2\pi^2}}{2\epsilon}$, see \cite{Eriksson}.
The exact solutions for $\epsilon \in \{0.01,0.001,0.0001\}$ are presented in Figure \ref{fig:second}.
}

\begin{figure}
    \centering
    \includegraphics[width=0.9\linewidth]{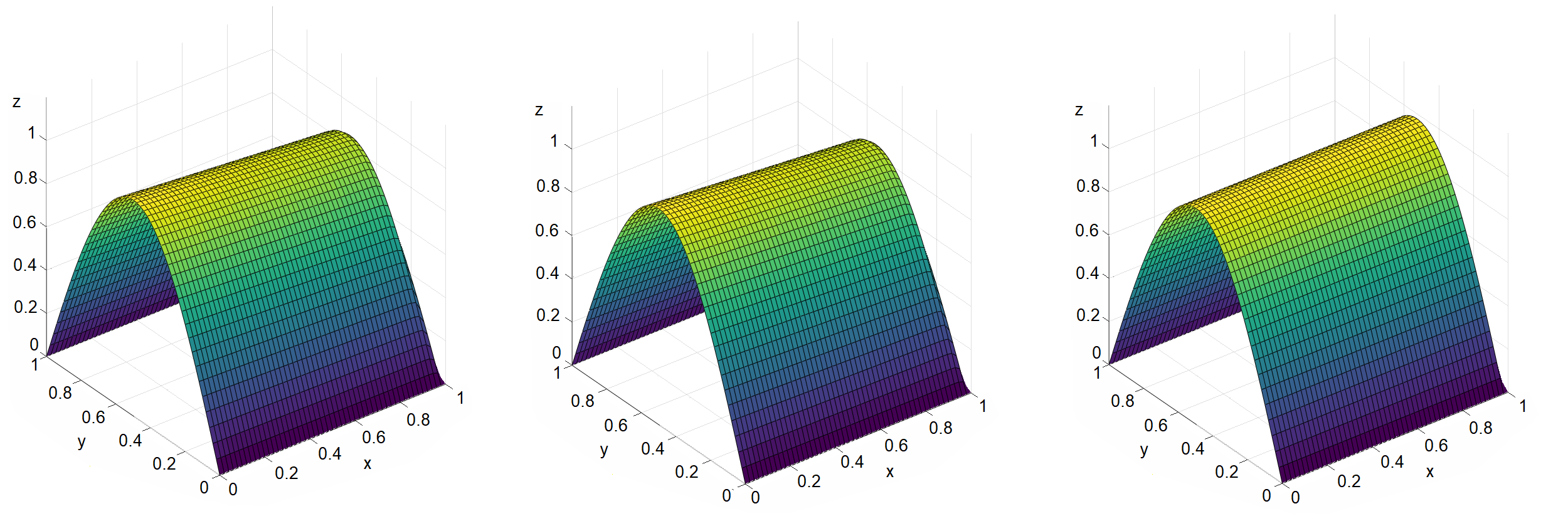}
    \caption{\revision{The exact solutions for the second problem for $\epsilon \in \{0.01, 0.001, 0.0001\}$}}
    \label{fig:second}
\end{figure}

\subsection{Bubnov-Galerkin method for \revision{advection-difussion problems}}

 We develop the standard weak formulation for the \revision{advection-diffusion model} problem:
Find $u \in H^1_g(\Omega)$
\begin{eqnarray}
\begin{aligned}
\left(\frac{\partial u}{\partial x},v\right)_{\Omega}+\epsilon \left(  \frac{\partial u}{\partial x}, \frac{\partial v}{\partial x}\right)_{\Omega} +\epsilon \left(  \frac{\partial u }{\partial y}, \frac{\partial v}{\partial y}\right)_{\Omega} = 0, \, \forall v \in H^1_0(\Omega).
\label{eq:ModelProblem_b}
\end{aligned}
\end{eqnarray}

\subsection{Least-squares finite element method
for \revision{advection-diffusion} model problem}

Now, we present the discrete weak form of (\ref{eq:ModelProblem_b}) in the spirit of the least-squares finite element method
using optimal test functions. Hence, we seek $u_h \in U^h \subset C^2\left(\Omega\right)$ such that:
\begin{eqnarray}
\begin{aligned}
B(u_h,v_h)=&
\left( 
 \frac{\partial u_h}{\partial x}
-\epsilon   \frac{\partial^2 u_h}{\partial x^2}
-\epsilon  \frac{\partial^2 u_h}{\partial y^2}
,v_h \right)
= 0, \quad \forall v_h \in \revision{V_h} 
\label{eq:ModelProblem_b2}
\end{aligned}
\end{eqnarray}
\revision{where $V_h \subset L_2(\Omega)$ to be determined.}
For the advection-dominated \revision{advection-}diffusion problem, stability in the Bubnov-Galerkin setup is not guaranteed. 
\revision{It means, that the numerical solution may not converge to the exact solution when we increase accuracy by refining the mesh.}
The solution we rely on to ensure stability  is to apply a Petrov-Galerkin method with optimal test functions. To this end, we define a set of optimal test functions for each trial function that are the solutions of auxiliary weak problems. 

The approximate solution $u_h$ is a linear combination of the trial basis functions, and some unknown coefficients $u_h^i$:
\begin{equation}
u_h = \sum_{i}u^i_h e^i,
\end{equation}
where $e^i$ are the basis functions. Following the philosophy of optimal testing in the DPG method, for each basis function $e^i$, there exists an optimal test function that can be computed in the following way:

Find \revision{$\hat{e}^i \in L_2\left(\Omega\right)$} such that 
\begin{equation}
\left(\hat{e}^i,\psi\right)_{\Omega} = B\left(e^i,\psi\right) \quad \forall \psi \in \revision{L_2\left(\Omega\right)},
\label{eq:optimall}
\end{equation}
\revision{Since~$L_2(\Omega)$ is a Hilbert space
and~$\psi \mapsto B(e_i, \psi)$ defines a (continuous) linear functional on~$L_2(\Omega)$,
existence and uniqueness of such~$\hat{e}_i$ is guaranteed by the Riesz representation theorem.}
Note that there is one problem (\ref{eq:optimall}) for each trial basis function in $U_h$. 

Continuing this reasoning, 
we have:
\begin{equation}
\left( \hat{e}^i,\psi \right)_{\Omega} = \left( 
 \frac{\partial \revision{e^i}}{\partial x}
-\epsilon   \frac{\partial^2 \revision{e^i}}{\partial x^2}
-\epsilon  \frac{\partial^2 \revision{e^i}}{\partial y^2}
,\psi\right)_{\revision{\Omega}} \quad \forall \psi \in \revision{L_2(\Omega)} 
\end{equation}
which means that:
\begin{equation}
\left(\hat{e}^{i}-\left[ \frac{\partial \revision{e^i}}{\partial x}
-\epsilon   \frac{\partial^2 \revision{e^i}}{\partial x^2}
-\epsilon  \frac{\partial^2 \revision{e^i}}{\partial y^2}\right]
,\psi\right)_{\Omega} = 0\quad \forall \psi \in \revision{L_2(\Omega)} 
\label{eq:l2}
\end{equation}
Since (\ref{eq:l2}) holds for all $\psi$ we must conclude from the Fourier Lemma that:
\begin{equation}
\hat{e}^{i}-\left[ \frac{\partial \revision{e^i}}{\partial x}
-\epsilon   \frac{\partial^2 \revision{e^i}}{\partial x^2}
-\epsilon  \frac{\partial^2 \revision{e^i}}{\partial y^2}
\right]=0
\end{equation}
\begin{equation}
\hat{e}^{i}= \frac{\partial \revision{e^i}}{\partial x}
-\epsilon   \frac{\partial^2 \revision{e^i}}{\partial x^2}
-\epsilon  \frac{\partial^2 \revision{e^i}}{\partial y^2}
\end{equation}

The optimal test function for the trial basis function $e^i$ is the action of the differential operator onto the basis function from the trial space. Since this operation involves second-order derivatives, higher-order discretization, like the isogeometric analysis, is needed.
Since we now have an explicit (analytic) relationship for each optimal test function, we can show that we have unconditional stability. The adjusted weak form with optimal testing is the least-squares weak form:

Find $u_h\in U_h$:
\begin{eqnarray}
\begin{aligned}
& \left(\frac{\partial u_h}{\partial x}
-\epsilon   \frac{\partial^2 u_h}{\partial x^2}
-\epsilon  \frac{\partial^2 u_h}{\partial y^2}
,\underbrace{ \frac{\partial v_h}{\partial x}
-\epsilon   \frac{\partial^2 v_h}{\partial x^2}
-\epsilon  \frac{\partial^2 v_h}{\partial y^2}
}_{w_h}\right)_{\Omega} = \\ & = \left(f,\underbrace{ \frac{\partial v_h}{\partial x}
-\epsilon   \frac{\partial^2 v_h}{\partial x^2}
-\epsilon  \frac{\partial^2 v_h}{\partial y^2}
}_{w_h}\right)_{\Omega} \quad \forall v_h \in U^h
\label{eq:weak2}
\end{aligned}
\end{eqnarray}
Notice that in (\ref{eq:weak2}) we have $v_h \in U_h$. 
The coercivity of the formulation follows trivially from this choice of test functions and can also be found for general least-squares finite element methods in, e.g.,~\cite{lsfem}.

\subsubsection{Bubnov-Galerkin method on a manually adapted grid}

As a first numerical experiment to highlight the strength of mesh adaptation, we solve this model problem on a manually refined grid using quadratic \revision{$C^1$ (no inner knot value repetition)} B-splines. They are defined with the following knot vectors and points:

{\tt knot\_x  = [0 0 0 1 2 3 4 5 6 7 8 9 10 11 12 13 14 15 16 17 18 19 20 21 22 23 24 25 26 26 26];}

{\tt points\_x = [0 0.5 0.75 0.875 0.9375 0.96875 0.984375 0.9921875 0.99609375 0.998046875 0.9990234375 0.9995117188 0.9997558594 0.9998779297 0.9999389648 0.9999694824 0.9999847412 0.9999923706 0.9999961853 0.9999980927 0.9999990463 0.9999995232 0.9999997616 0.9999998808 0.9999999404 0.9999999702 1];}

{\tt knot\_y = [0 0 0 1 2 3 4 5 6 7 8 9 10 11 12 13 14 15 16 17 18 19 20 21 22 23 24 25 26 26 26];}

{\tt points\_y = [0 0.5 0.75 0.875 0.9375 0.96875 0.984375 0.9921875 0.99609375 0.998046875 0.9990234375 0.9995117188 0.9997558594 0.9998779297 0.9999389648 0.9999694824 0.9999847412 0.9999923706 0.9999961853 0.9999980927 0.9999990463 0.9999995232 0.9999997616 0.9999998808 0.9999999404 0.9999999702 1];}

\begin{figure}
\centering
\includegraphics[scale=0.2]{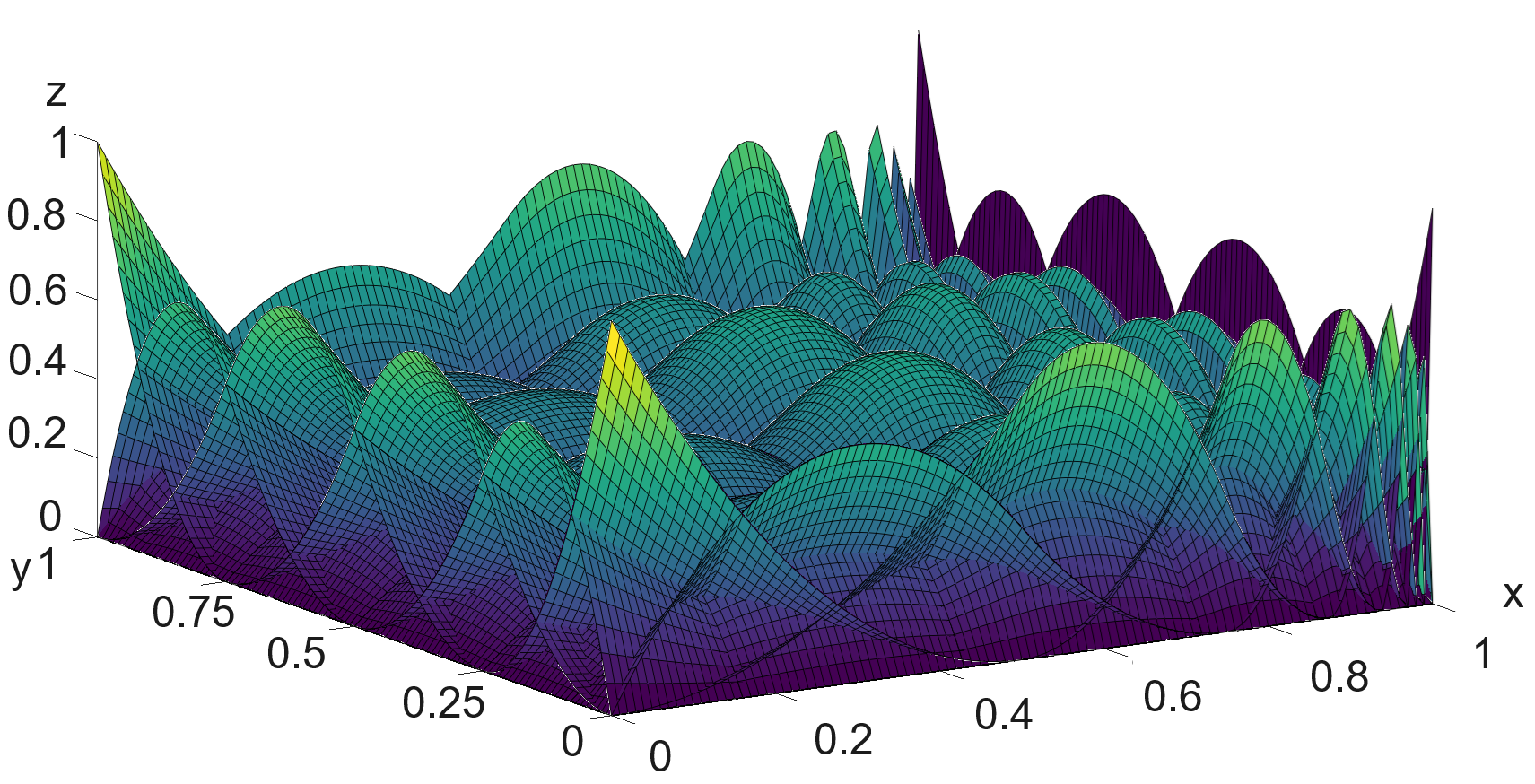}
\caption{\revision{B-spline basis functions for non-uniform grid for the Eriksson-Johnsons problem.
}}
\label{fig:nonuniform}
\end{figure}

\revision{The basis functions along $x$ axis are obtained by introducing knot points $\xi_i = points\_x[knot\_x[i]+1]$ into the recursive formula (\ref{eq:bsplines}),
\begin{eqnarray}
\begin{aligned}
    B_{i,0}(\xi) & =1 \textrm{ if } \xi_i \leq \xi \leq \xi_{i+1}, \textrm{ otherwise } 0, \\
    B_{i,p}(\xi) &=\frac{\xi - \xi_i}{\xi_{i+p}-\xi_i} B_{i,p-1}(\xi) + \frac{\xi_{i+p+1}-\xi}{\xi_{i+p+1}-\xi_i} B_{i+1,p-1}(\xi) 
\end{aligned}
\label{eq:bsplines}
\end{eqnarray}
for the order $p$ defined as the number of repetitions of the first $knot\_x[1]$ minus one, assuming that the subsequent knots inserted into the denominator must be different, and if they are not different, then the given term is changed to zero. The basis functions along the $y$ axis are constructed in a similar way.}

\revision{For the first model problem, we define two-dimensional basis functions as a tensor product of these basis functions. For the second numerical problem, we take the tensor product of basis functions defined by {\tt knot\_x} and {\tt points\_x} with the basis functions defined as}

{\tt knot\_y  = [0 0 0 1 2 3 4 4 4];} {\tt points\_y = [0 0.25 0.5 0.75 1];}

\revision{We perform numerical  experiments for different diffusion parameters: $\epsilon \in \{0.1,0.01,0.003\}$ for the first model problem and $\epsilon \in \{0.01,0.001,0.0001\}$ for the second problem.
We compare the numerical solutions for these parameters with the corresponding exact solutions in Figures \ref{fig:Galerkin001A}-\ref{fig:Galerkin001}.
We summarize the relative $L_2$ and $H^1$ errors in Table \ref{tab:tab1} and notice that the manually adapted grid allows for an accurate solution of this problem independent of the diffusivity.}

\begin{figure}
\centering
\includegraphics[scale=0.3]{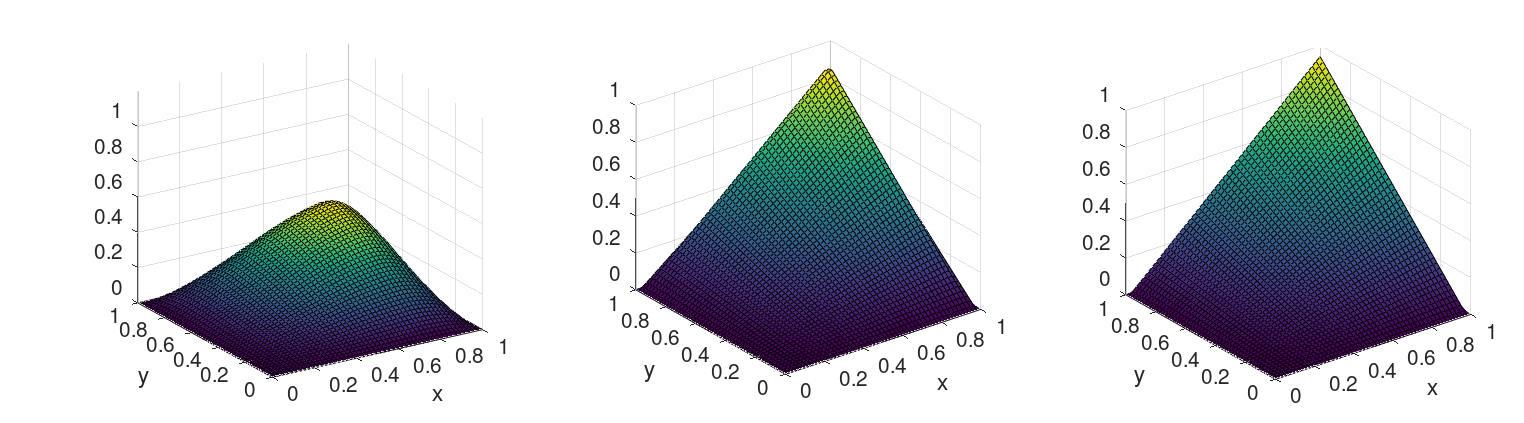}
\caption{\revision{The first model problem, for $\epsilon \in \{ 0.1, 0.01, 0.003 \}$. The solution of the Bubnov-Galerkin problem on a manually refined grid.}
}
\label{fig:Galerkin001A}
\end{figure}

\begin{figure}
\centering
\includegraphics[scale=0.3]{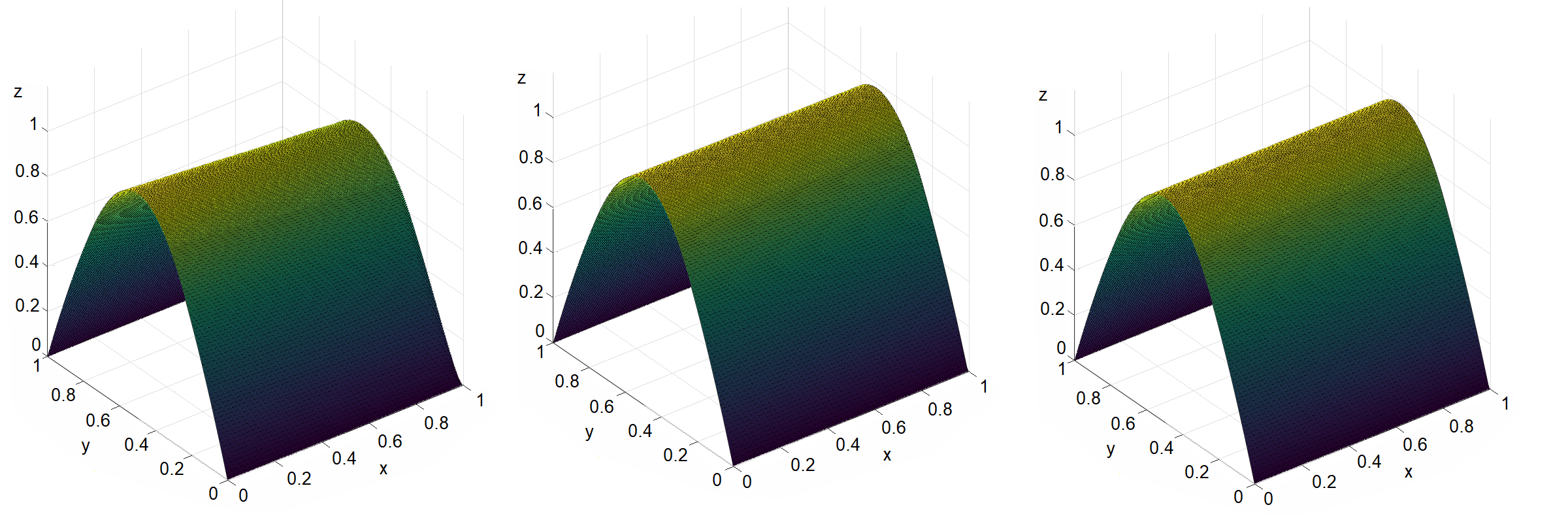}
\caption{\revision{The second model problem, Eriksson-Johnson for $\epsilon \in \{0.01, 0.001, 0.0001 \}$. The solution of the Bubnov-Galerkin problem on a manually refined grid.}
}
\label{fig:Galerkin001}
\end{figure}

\begin{center}
\begin{table}[h!]
\begin{tabular}{  | c |c | c|} 
\hline
 $\epsilon$ &  100$\frac{\|u-u_{exact}\|_{L_2}}{\|u_{exact}\|_{L_2}}$ & 100$\frac{\|u-u_{exact}\|_{H^1}}{\|u_{exact}\|_{H^1}}$ \\
  \hline
$\epsilon=0.1$ & 0.49 & 2.60 \\ 
$\epsilon=0.01$ & 0.12 & 2.30 \\ 
$\epsilon=0.003$ & 0.07 & 2.32 \\ 
  \hline
\end{tabular}
\caption{\revision{The first model problem. Numerical accuracy of solution of the Bubnov-Galerkin method on a manually refined grid.}}
\label{tab:tab1A}
\end{table}
\end{center}

\begin{center}
\begin{table}[h!]
\begin{tabular}{  | c |c | c|} 
\hline
 $\epsilon$ &  100$\frac{\|u-u_{exact}\|_{L_2}}{\|u_{exact}\|_{L_2}}$ & 100$\frac{\|u-u_{exact}\|_{H^1}}{\|u_{exact}\|_{H^1}}$ \\
  \hline
$\epsilon=0.01$ & 0.3 & 2.30 \\ 
$\epsilon=0.001$ & 0.27 & 2.29 \\ 
$\epsilon=0.0001$ & 0.27 & 2.29 \\ 
  \hline
\end{tabular}
\caption{\revision{The second model problem, the} Eriksson-Johnson problem. Numerical accuracy of solution of the Bubnov-Galerkin method on a manually refined grid.}
\label{tab:tab1}
\end{table}
\end{center}

\subsubsection{Bubnov-Galerkin method on a uniform grid}

Next, to highlight the issues when using the Bubnov-Galerkin method on uniform meshes, we solve this problem using a uniform coarse mesh. Hence, we solve this problem without  using our pre-existing knowledge about the solution behaviour.
We again use quadratic~\revision{$C^1$} B-splines defined by the knot vectors and points:

\revision{{\tt  knot\_x  = [0 0 0 1 2 3 4 5 6 7 8 9 10 10 10];}}

\revision{{\tt points\_x = [0 0.1 0.2 0.3 0.4 0.5 0.6 0.7 0.8 0.9 1];}}

\revision{{\tt  knot\_y  = [0 0 0 1 2 3 4 5 6 7 8 9 10 10 10];}}

\revision{{\tt points\_y = [0 0.1 0.2 0.3 0.4 0.5 0.6 0.7 0.8 0.9 1];}}

\revision{for the first problem and}

\revision{{\tt  knot\_x  = [0 0 0 1 2 3 4 5 6 7 8 9 10 10 10];}}

\revision{{\tt points\_x = [0 0.1 0.2 0.3 0.4 0.5 0.6 0.7 0.8 0.9 1];}}

\revision{{\tt  knot\_y  = [0 0 0 1 2 3 4 4 4]; }}

\revision{{\tt  points\_y = [0 0.25 0.5 0.75 1];}}

\revision{for the second problem.}

\revision{They define the B-splines on the uniform mesh illustrated in Figure \ref{fig:uniform}.}

\begin{figure}
\centering
\includegraphics[scale=0.2]{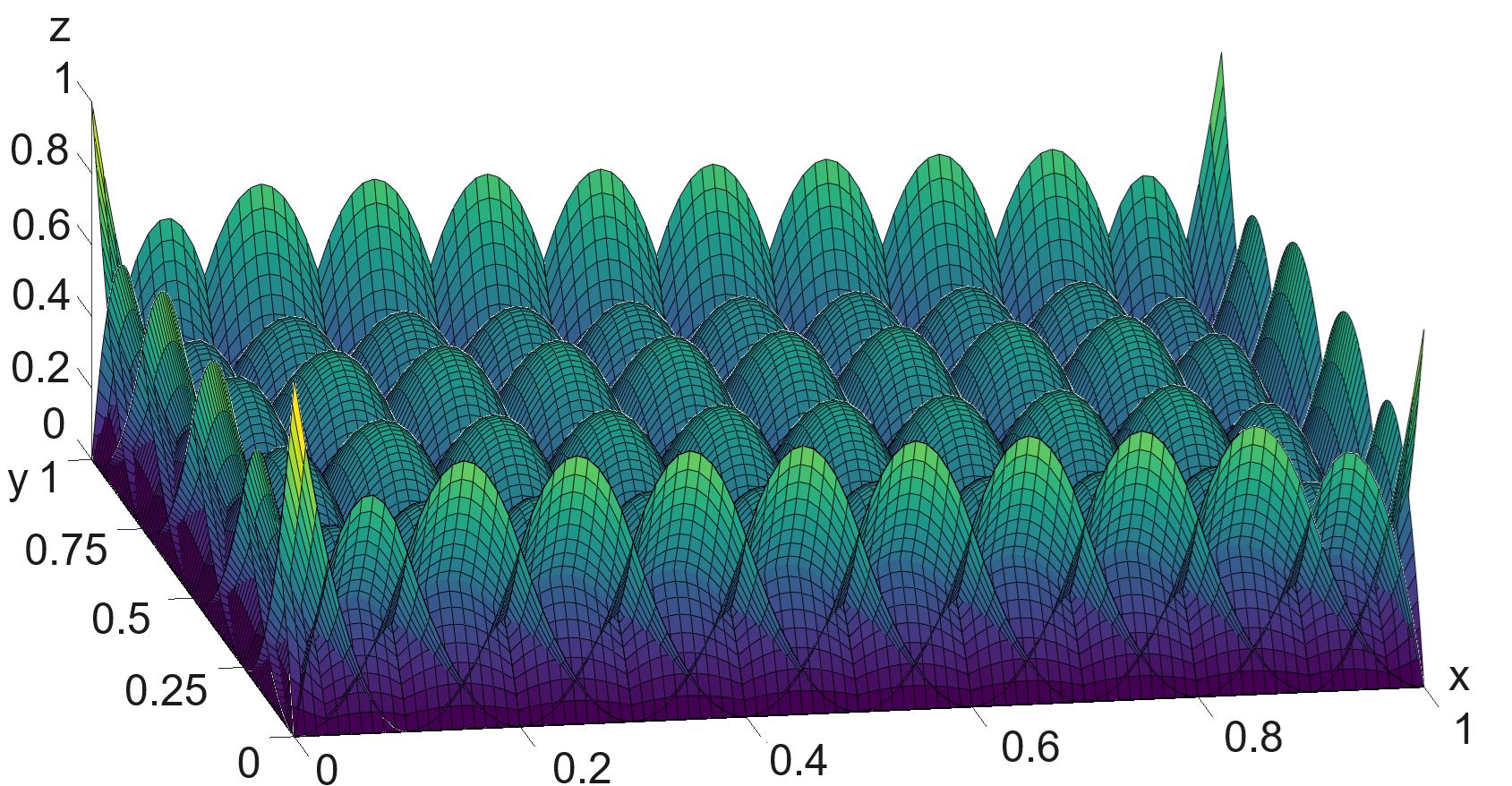}
\caption{\revision{B-spline basis functions for uniform grid for the Eriksson-Johnson problem.}}
\label{fig:uniform}
\end{figure}

\revision{We present the corresponding numerical results obtained for different diffusion parameters $\epsilon \in \{0.1,0.01,0.003\}$ for the first problem, and 
$\epsilon \in \{0.01,0.001,0.0001\}$ for the second problem.}
We present the numerical solutions alongside the exact solutions in Figures \ref{fig:GalerkinUniform01A}-\ref{fig:GalerkinUniform01}.
We summarize the results in Table \ref{tab:tab2} and clearly notice that the Bubnov-Galerkin method does not perform well on a uniform coarse mesh.

\begin{figure}
\centering
\includegraphics[scale=0.6]{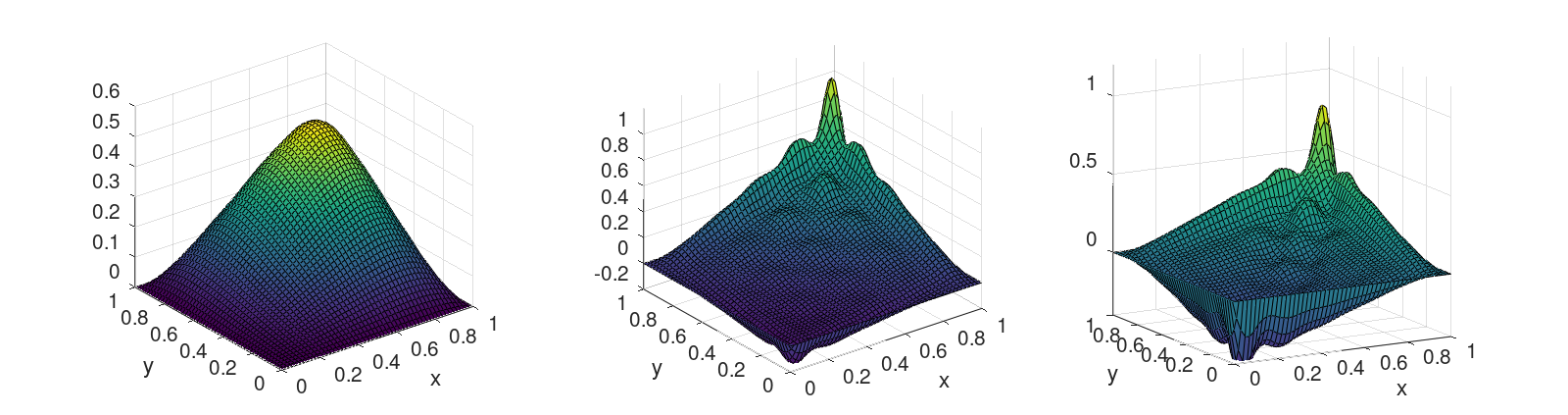}
\caption{\revision{The first model problem for $\epsilon \in \{0.1, 0.001, 0.0003\}$. The solution of the Bubnov-Galerkin problem on a uniform grid.}
}
\label{fig:GalerkinUniform01A}
\end{figure}

\begin{figure}
\centering
\includegraphics[scale=0.3]{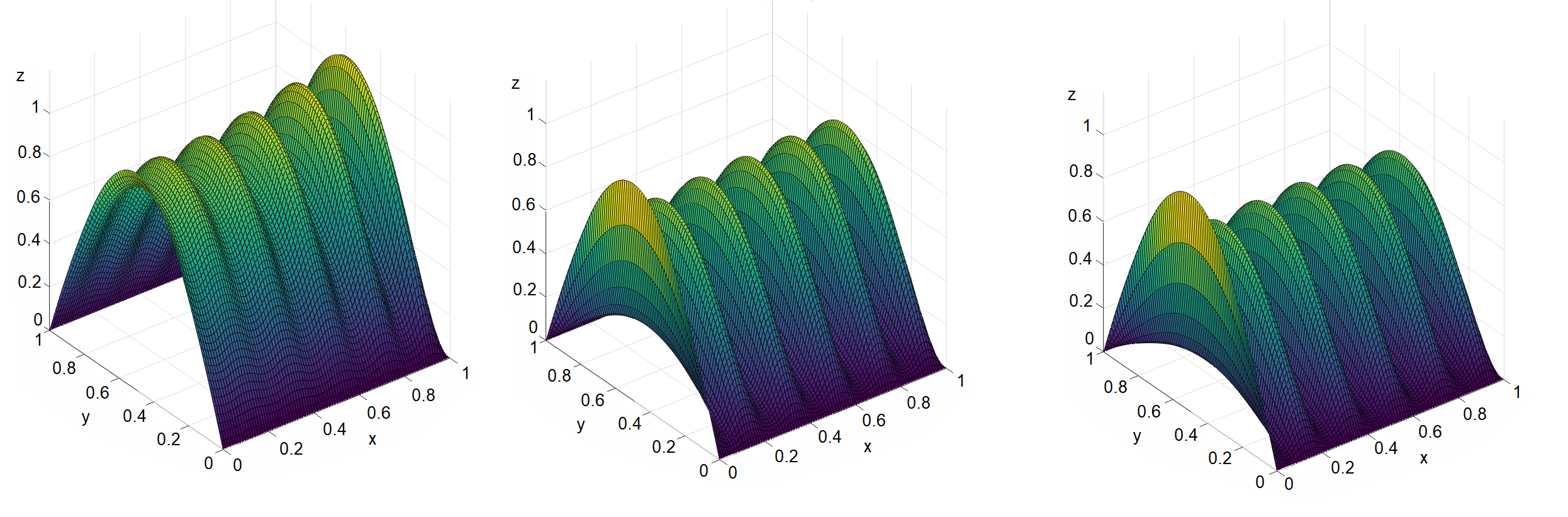}
\caption{\revision{The second model problem, Eriksson-Johnson for $\epsilon \in \{0.01,0.001,0.0001\} $. The solution of the Bubnov-Galerkin problem on a uniform grid.}
}
\label{fig:GalerkinUniform01}
\end{figure}

\begin{center}
\begin{table}
\begin{tabular}{  | c |c | c|} 
\hline
 $\epsilon$ &  100$\frac{\|u-u_{exact}\|_{L_2}}{\|u_{exact}\|_{L_2}}$ & 100$\frac{\|u-u_{exact}\|_{H^1}}{\|u_{exact}\|_{H^1}}$ \\
  \hline
$\epsilon=0.1$ & 0.60  &  4.69 \\ 
$\epsilon=0.01$ &  46.11 &  60.31 \\ 
$\epsilon=0.003$ & 87.17  & 189.62 \\ 
  \hline
\end{tabular}
\caption{\revision{The first model problem. Numerical accuracy of solution of the Bubnov-Galerkin method on a uniform grid.}}
\label{tab:tab2A}
\end{table}
\end{center}

\begin{center}
\begin{table}
\begin{tabular}{  | c |c | c|} 
\hline
 $\epsilon$ &  100$\frac{\|u-u_{exact}\|_{L_2}}{\|u_{exact}\|_{L_2}}$ & 100$\frac{\|u-u_{exact}\|_{H^1}}{\|u_{exact}\|_{H^1}}$ \\
  \hline
$\epsilon=0.01$ & 13.48 & 70.44 \\ 
$\epsilon=0.001$ & 48.15 & 259.75 \\ 
$\epsilon=0.0001$ & 54.77 & 262.14 \\ 
  \hline
\end{tabular}
\caption{\revision{The second model problem, Eriksson-Johnson problem}. Numerical accuracy of solution of the Bubnov-Galerkin method on a uniform grid.}
\label{tab:tab2}
\end{table}
\end{center}

\subsubsection{Petrov-Galerkin method with optimal test functions in $L_2$ setup on a uniform grid} \label{sec:petrov_optimal_eriksson}

\revision{Next, we apply the  Petrov-Galerkin formulation with optimal test functions on a uniform grid. We use the same basis functions and diffusion values as in the previous section. 
We present the corresponding numerical results using this formulation with optimal test functions in the $L_2$ setup in Figures \ref{fig:Eirik01A} -\ref{fig:Eirik01}.}
We can see from these figures, that the solution is indeed stable and less oscillatory than the Bubnov-Galerkin method in the previous section, but  appears overly diffusive for small values of $\epsilon$. In Tables \ref{tab:tab3A}-\ref{tab:tab3}, the corresponding relative errors are summarized.

\begin{figure}
\centering
\includegraphics[scale=0.3]{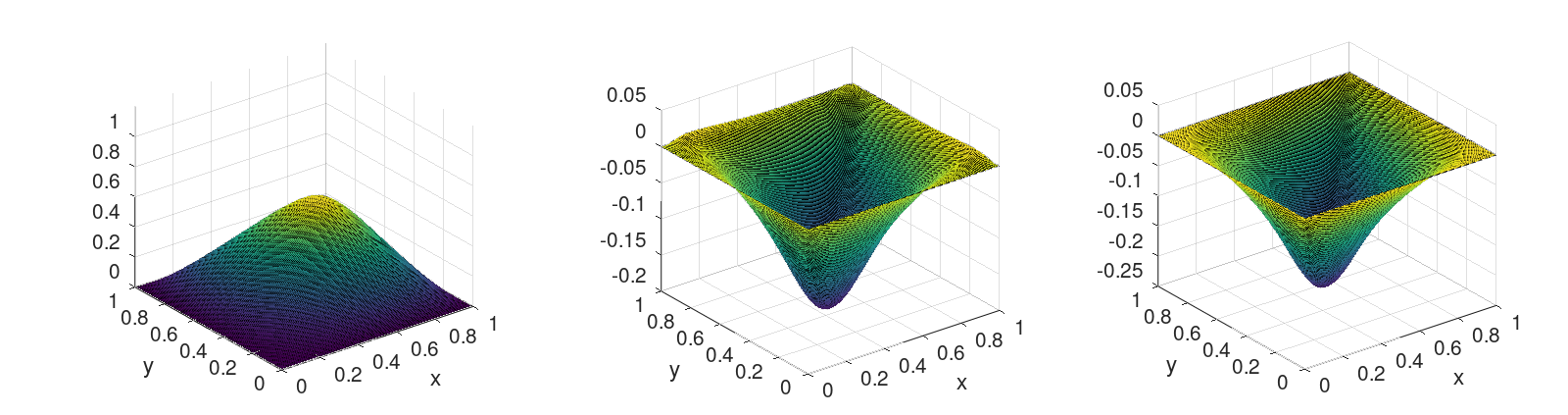}
\caption{\revision{The first model problem for $\epsilon \in \{0.1,0.01,0.003\}$. The solution of the Petrov-Galerkin with optimal test functions in $L_2$ setup.}
}
\label{fig:Eirik01A}
\end{figure}

\begin{figure}
\centering
\includegraphics[scale=0.65]{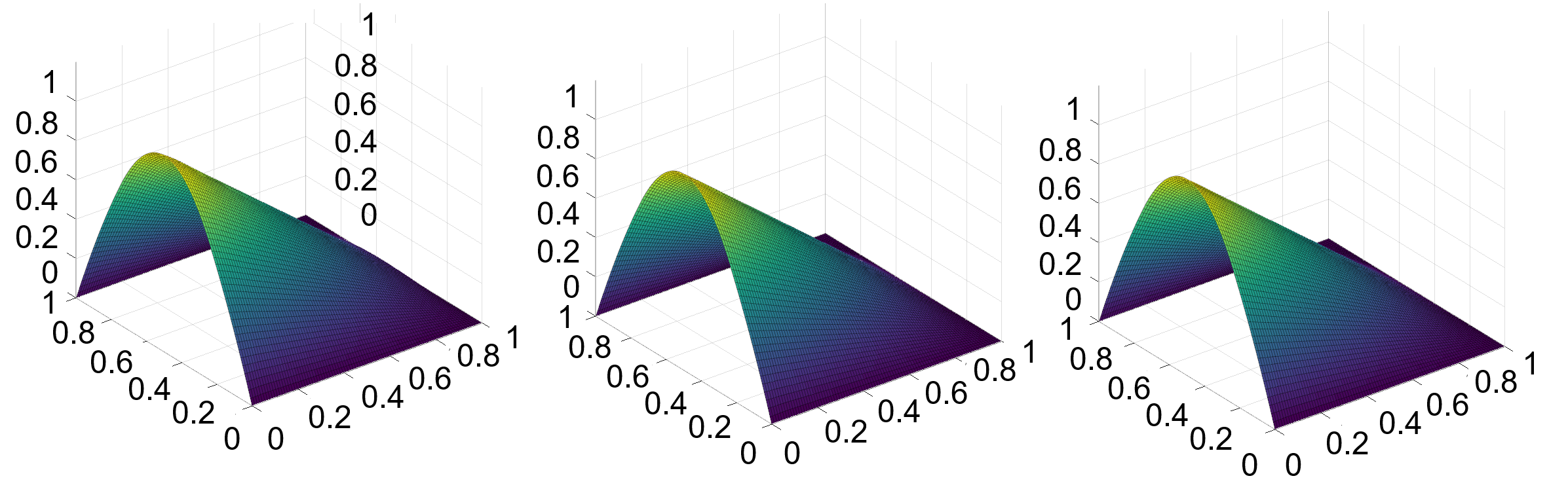}
\caption{\revision{Eriksson-Johnson for $\epsilon \in \{0.01,0.001,0.0001\}$. The solution of the Petrov-Galerkin with optimal test functions in $L_2$ setup.}}
\label{fig:Eirik01}
\end{figure}

\begin{center}
\begin{table}
\begin{tabular}{  | c |c | c|} 
\hline
 $\epsilon$ &  100$\frac{\|u-u_{exact}\|_{L_2}}{\|u_{exact}\|_{L_2}}$ & 100$\frac{\|u-u_{exact}\|_{H^1}}{\|u_{exact}\|_{H^1}}$ \\
  \hline
$\epsilon=0.1$ & 17.27 & 16.61 \\ 
$\epsilon=0.01$ & 86.70 & 100.83 \\ 
$\epsilon=0.003$ & 84.43 & 100.80 \\ 
  \hline
\end{tabular}
\caption{\revision{The first model problem. Numerical accuracy of the solution with the optimal test functions in $L_2$ setup on uniform grid.}}
\label{tab:tab3A}
\end{table}
\end{center}

\begin{center}
\begin{table}
\begin{tabular}{  | c |c | c|} 
\hline
 $\epsilon$ &  100$\frac{\|u-u_{exact}\|_{L_2}}{\|u_{exact}\|_{L_2}}$ & 100$\frac{\|u-u_{exact}\|_{H^1}}{\|u_{exact}\|_{H^1}}$ \\
  \hline
$\epsilon=0.01$ & \revision{52.87} & \revision{86.47} \\ 
$\epsilon=0.001$ & \revision{57.36} & \revision{64.83} \\ 
$\epsilon=0.0001$ & \revision{57.70} & \revision{65.19} \\ 
  \hline
\end{tabular}
\caption{\revision{The second model problem, Eriksson-Johnson}. Numerical accuracy of the solution with the optimal test functions in $L_2$ setup on uniform grid.}
\label{tab:tab3}
\end{table}
\end{center}

\subsection{Galerkin/least\revision{-}squares formulation }

To overcome the issue we noted in Section~\ref{sec:petrov_optimal_eriksson} with overly diffusive solutions for small $\epsilon$ we now introduce a modified stabilization. In this case, we adjust the selection of the test fnctions inspired by the Galerkin/least squares method \cite{GLS}, and the AVS-FE method which adds weighted contributions of derivatives to the $L_2$ norm in the minimization process~\cite{avs}.
Let us show this choice on the model \revision{advection-diffusion} problem as we did before.


We take the strong form of the PDE $
\frac{\partial u}{\partial x}-\epsilon \left(\frac{\partial^2 u}{\partial x^2}+
\frac{\partial^2 u}{\partial y^2}\right)= 0$ and we test with $a+\frac{1}{h}b$:
 
\begin{equation}
\left(\frac{\partial u}{\partial x}-\epsilon \left(\frac{\partial^2 u}{\partial x^2}+
\frac{\partial^2 u}{\partial y^2}\right),a+\frac{1}{h} b \right)=0,
\end{equation}

\noindent where for $a$ we take the optimal test functions in $L_2$ setup, as derived in Section 2, and for $b$ we take the test functions equal to 0 on the boundary. 

\revision{We use the same basis functions and diffusion coefficients as in the previous section, namely, we define quadratic B-splines of $C^1$ continuity.} We remove the first and the last basis function, the two that are non-zero on the boundary on the domain.

\begin{equation}
\left(\frac{\partial u}{\partial x}-\epsilon \left(\frac{\partial^2 u}{\partial x^2}+
\frac{\partial^2 u}{\partial y^2}\right),a\right)+
\left(\frac{\partial u}{\partial x}-\epsilon \left(\frac{\partial^2 u}{\partial x^2}+
\frac{\partial^2 u}{\partial y^2}\right),\frac{1}{h} b \right)=0 
\end{equation}

\noindent We then integrate by parts the second term:

\begin{equation}
\left(\frac{\partial u}{\partial x}-\epsilon \left(\frac{\partial^2 u}{\partial x^2}+
\frac{\partial^2 u}{\partial y^2}\right),a \right)+\frac{1}{h}\left(\epsilon \frac{\partial u}{\partial x},\frac{\partial b}{\partial x}\right)+
\frac{1}{h}\left(\epsilon \frac{\partial u}{\partial y},\frac{\partial b}{\partial y}\right)+
\frac{1}{h}\left(\frac{\partial u}{\partial x},b\right)=0 +b.c.
\end{equation}


With this modified  Petrov-Galerkin formulation with stabilization, we perform numerical experiments to showcase its properties. 
We present numerical results in Figures \ref{fig:Improved01A}-\ref{fig:Improved01}.
In Tables \ref{tab:tab4A}-\ref{tab:tab4}, the resulting relative errors are summarized for these cases. 
\revision{Even for low value of the $\epsilon=0.003$ for the first model problem or $\epsilon=0.0001$ for the second model problem}, the solution has a good shape on the mesh with only ten uniform elements as shown in Figures~\ref{fig:Improved01A}-\ref{fig:Improved01}.

\begin{figure}
\centering
\includegraphics[scale=0.6]{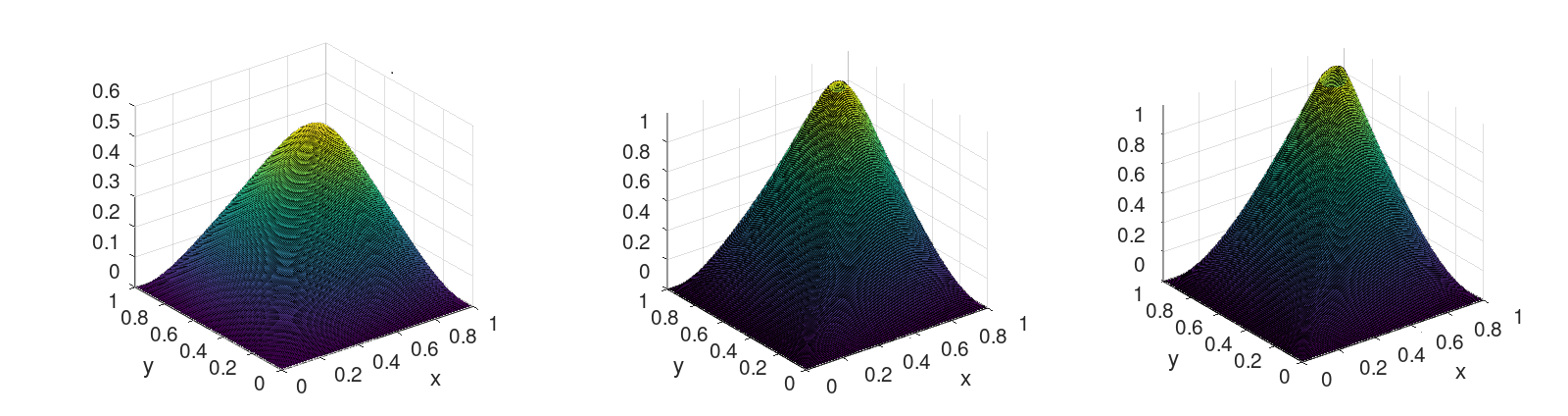}
\caption{\revision{The first model problem for $\epsilon=\{0.1, 0.01, 0.003\}$. The solution with the modified stabilization on a uniform mesh of ten elements.}
}
\label{fig:Improved01A}
\end{figure}

\begin{figure}
\centering
\includegraphics[scale=0.4]{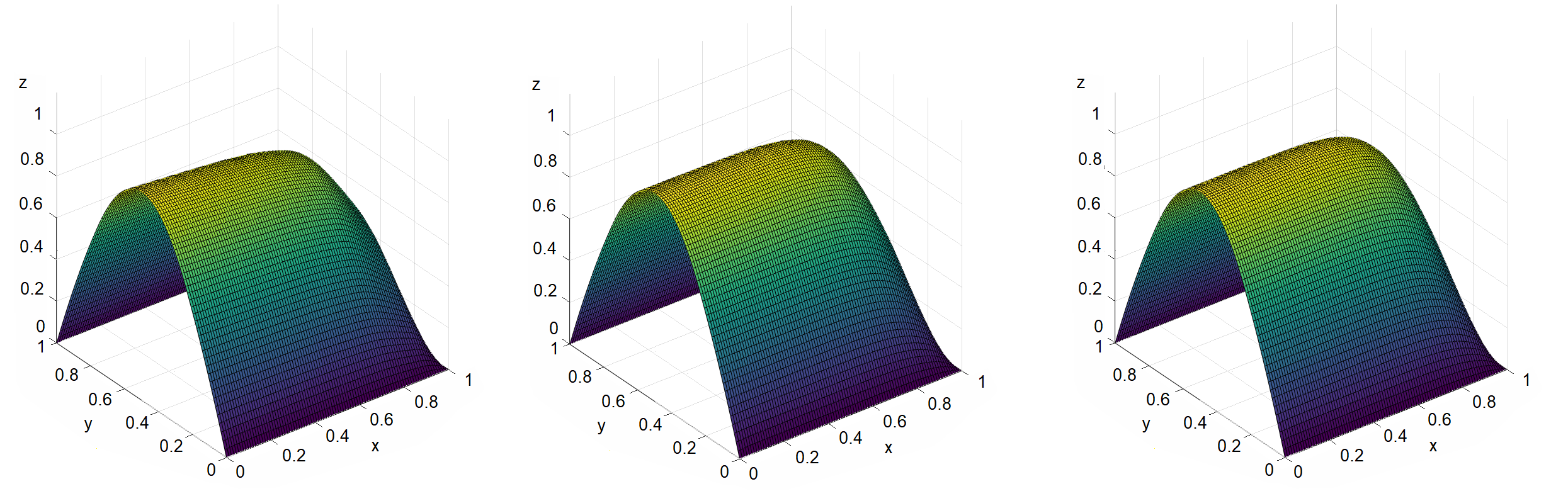}
\caption{\revision{The second model problem, Eriksson-Johnson for $\epsilon \in \{0.01,0.001,0.0001\}$. The solution with the modified stabilization on a uniform mesh of ten elements.}
}
\label{fig:Improved01}
\end{figure}

\begin{center}
\begin{table}
\begin{tabular}{  | c |c | c|} 
\hline
 $\epsilon$ &  100$\frac{\|u-u_{exact}\|_{L_2}}{\|u_{exact}\|_{L_2}}$ & 100$\frac{\|u-u_{exact}\|_{H^1}}{\|u_{exact}\|_{H^1}}$ \\
  \hline
$\epsilon=0.1$ & 1.64  & 4.73 \\ 
$\epsilon=0.01$ & 28.20  & 63.04 \\ 
$\epsilon=0.003$ &  38.15 & 135.54  \\ 
  \hline
  \end{tabular}
\caption{\revision{The first model problem. Numerical accuracy of solution with modified stabilization  on uniform grid of ten elements.}}
\label{tab:tab4A}
\end{table}
\end{center}

\begin{center}
\begin{table}
\begin{tabular}{  | c |c | c|} 
\hline
 $\epsilon$ &  100$\frac{\|u-u_{exact}\|_{L_2}}{\|u_{exact}\|_{L_2}}$ & 100$\frac{\|u-u_{exact}\|_{H^1}}{\|u_{exact}\|_{H^1}}$ \\
  \hline
$\epsilon=0.1$ & 0.91 & 3.11 \\ 
$\epsilon=0.01$ &  17.10 & 62.44  \\ 
$\epsilon=0.001$ &  22.17 & 71.33 \\ 
$\epsilon=0.0001$ &  22.38 & 71.46 \\ 
  \hline
  \end{tabular}
\caption{\revision{The second model problem,} Eriksson-Johnson problem. Numerical accuracy of solution with modified stabilization  on uniform grid of ten elements.}
\label{tab:tab4}
\end{table}
\end{center}

\subsection{The Streamline-Upwind Petrov-Galerkin method}

A commonly stabilization method is the SUPG method \cite{SUPG}.
We will explain it for the \revision{advection-diffusion}  model problems we consider.
In the SUPG method, we modify the discretized weak form by adding the residual term acting upon the discrete solution $u_h$ in the following way:

\begin{eqnarray}
\begin{aligned}
& \beta_x \left(\frac{\partial u_h}{\partial x},v_h\right)
+  \beta_y \left(\frac{\partial u_h}{\partial y},v_h\right)
+\epsilon \left(  \frac{\partial u_h}{\partial x}, \frac{\partial v_h}{\partial x}\right) +\epsilon \left(  \frac{\partial u_h }{\partial y}, \frac{\partial v_h}{\partial y}\right) \\
& +
(R(u_h),\tau \beta\cdot \nabla v_h)
 = 0 + b.c.
\label{eq:Eriksson_b_SUPG}
\end{aligned}
\end{eqnarray}
where $R(u)=\frac{\partial u}{\partial x}+\epsilon \Delta u$, and $\tau^{-1}=
\left(\frac{\beta_x}{h_x} + \frac{\beta_y}{h_y} \right) + 3\epsilon \frac{1}{h_x^2+h_y^2}$, where in our case we run the simulations for the diffusion parameter \revision{$\epsilon \in \{0.1,0.01,0.003,0.001,0.0001\}$}, and the convection vector $\beta = (\beta_x,\beta_y)$, and $h_x$ and $h_y$ are the horizontal and the vertical dimensions of an element. 

\subsubsection{Streamline Upwind Petrov Galerkin method on uniform grid}

We use the same uniform mesh of the previous section. 
The summary of the numerical results obtained with the SUPG method on this uniform mesh is presented in Table \ref{tab:tab5}. The shape of the numerical solutions for $\epsilon \in \{0.1,0.01,0.003, 0.001,0.0001 \}$, are presented in Figures \ref{fig:SUPG01A}-\ref{fig:SUPG01}.

\begin{center}
\begin{table}[h]
\begin{tabular}{  | c |c | c|} 
\hline
 $\epsilon$ &  100$\frac{\|u-u_{exact}\|_{L_2}}{\|u_{exact}\|_{L_2}}$ & 100$\frac{\|u-u_{exact}\|_{H^1}}{\|u_{exact}\|_{H^1}}$ \\
  \hline
$\epsilon=0.1$ & 3.57 & 7.81 \\ 
$\epsilon=0.01$ & 33.71  & 71.91  \\ 
$\epsilon=0.003$ & 46.13  & 120.14 \\ 
  \hline
  \end{tabular}
\caption{\revision{The first model problem. Numerical accuracy of solution with the SUPG stabilization on a uniform grid of ten elements.}}
\label{tab:tab5A}
\end{table}
\end{center}

\begin{center}
\begin{table}[h]
\begin{tabular}{  | c |c | c|} 
\hline
 $\epsilon$ &  100$\frac{\|u-u_{exact}\|_{L_2}}{\|u_{exact}\|_{L_2}}$ & 100$\frac{\|u-u_{exact}\|_{H^1}}{\|u_{exact}\|_{H^1}}$ \\
  \hline
$\epsilon=0.01$ &  20.88 & 68.46  \\ 
$\epsilon=0.001$ &  22.45 & 71.11 \\ 
$\epsilon=0.0001$ &  22.44 & 71.78 \\ 
  \hline
  \end{tabular}
\caption{\revision{The second model problem, Eriksson-Johnson}. Numerical accuracy of solution with the SUPG stabilization on a uniform grid of ten elements.}
\label{tab:tab5}
\end{table}
\end{center}

\begin{figure}
\centering
\includegraphics[scale=0.6]{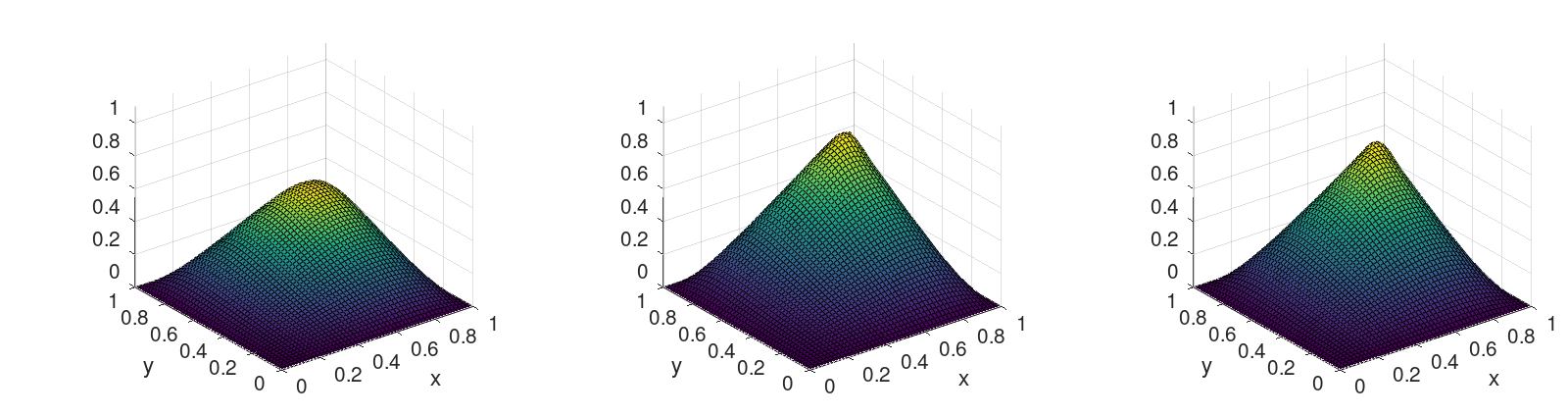}
\caption{\revision{The first model problem for $\epsilon \in \{0.1,0.001,0.003\}$. The solution with the SUPG stabilization on a uniform mesh of ten elements.}
}
\label{fig:SUPG01A}
\end{figure}

\begin{figure}
\centering
\includegraphics[scale=0.3]{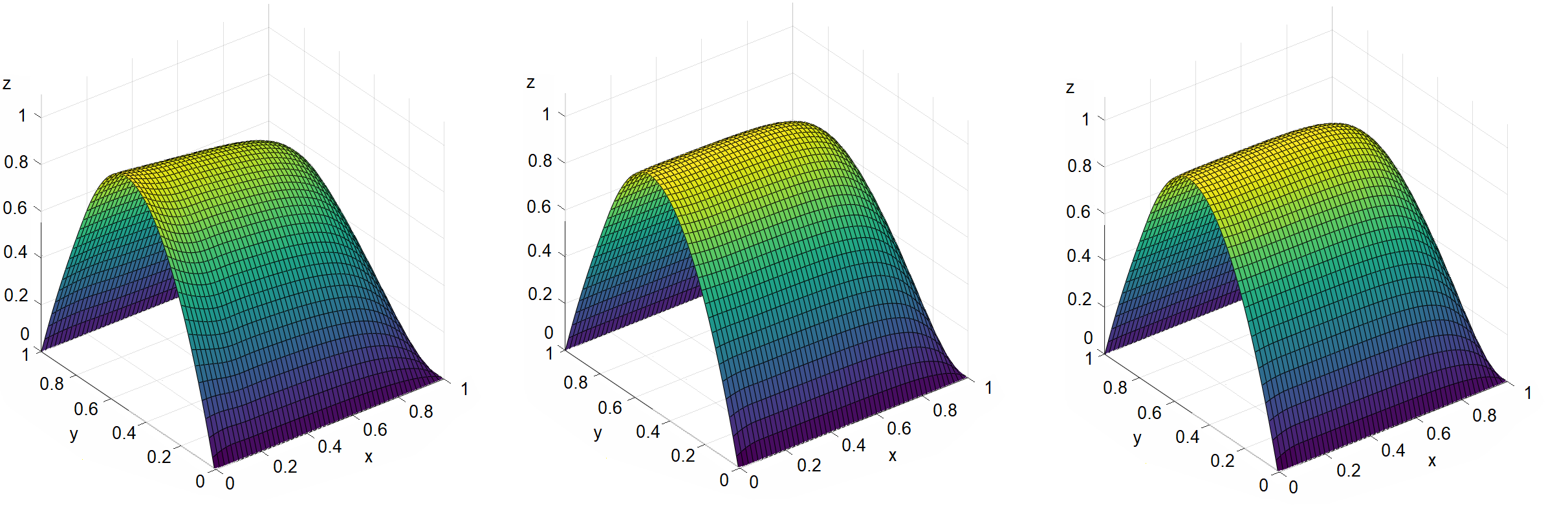}
\caption{\revision{The second model problem, Eriksson-Johnson for $\epsilon \in \{0.01, 0.001, 0.0001 \}$. The solution with the SUPG stabilization on a uniform mesh of  ten elements. }}
\label{fig:SUPG01}
\end{figure}

\section{\revision{Comparison of the methods}}

\begin{center}
\begin{table}[h]  
\begin{tabular}{  c|cc|cccccccc } 
\hline
& \multicolumn{2}{c}{Galerkin} & \multicolumn{2}{c}{Galerkin} & \multicolumn{2}{c}{Petrov-Galerkin} & 
\multicolumn{2}{c}{Galerkin} & \\
& \multicolumn{2}{c}{Refined} & \multicolumn{2}{c}{Uniform} & \multicolumn{2}{c}{Optimal} &
\multicolumn{2}{c}{Least-Sqaures} &
\multicolumn{2}{c}{SUPG} \\
\hline
\hline
 $\epsilon$ &  $L_2$ & $H_1$
 &   $L_2$ & $H_1$
 &   $L_2$ & $H_1$
 &  $L_2$ & $H_1$
 &  $L_2$ & $H_1$
 \\
  \hline
$0.1$ & 0.49 & 2.60 & \inred{0.6} & 4.69 & {17.27} & {16.61} & {1.64} & {4.73} & 3.57 & 7.81 \\ 
$0.01$ & 0.12 & 2.30 & 46.11 & 60.31 & {86.70} & {100.83} &  \inred{{28.20}}  & {63.04} & 33.71 & 71.91 \\ 
$0.003$ & 0.07 & 2.32 & 87.17 & 189.62 & {84.43}  & {100.80} & \inred{{38.15}}  & {135.54} & 46.13 & 120.14 \\ 
  \hline
  \end{tabular}
\caption{\revision{Comparison of the methods for the first model problem. $L_2=100\frac{\|u-u_{exact}\|_{L_2}}{\|u_{exact}\|_{L_2}}$,  $H_1=100\frac{\|u-u_{exact}\|_{H_1}}{\|u_{exact}\|_{H_1}}$}. }
\label{tab:comparison_first}
\end{table}
\end{center}

\begin{center}
\begin{table}[h]
\begin{tabular}{  c|cc|cccccccc } 
\hline
& \multicolumn{2}{c}{Galerkin} & \multicolumn{2}{c}{Galerkin} & \multicolumn{2}{c}{Petrov-Galerkin} & 
\multicolumn{2}{c}{Galerkin} & \\
& \multicolumn{2}{c}{Refined} & \multicolumn{2}{c}{Uniform} & \multicolumn{2}{c}{Optimal} &
\multicolumn{2}{c}{Least-Sqaures} &
\multicolumn{2}{c}{SUPG} \\
\hline
\hline
 $\epsilon$ &  $L_2$ & $H_1$
 &   $L_2$ & $H_1$
 &   $L_2$ & $H_1$
 &  $L_2$ & $H_1$ 
 &  $L_2$ & $H_1$ \\
  \hline
    $0.01$ & 0.3 & 2.30 &  \inred{13.48} & 70.44 & {52.87} & {86.47} & 17.10 & 62.44 & 20.88 & 68.46 \\ 
$0.001$ & 0.27 & 2.29 & 48.15 & 259.75 & {57.36} & {64.83} &  \inred{22.17} & 71.33 & 22.45 & 71.11 \\ $0.0001$ & 0.27 & 2.29 & 54.77 & 262.14 & {57.70} & {65.19} & \inred{22.38} & 71.46 & 22.44 & 71.78  \\ 
  \hline
  \end{tabular}
\caption{\revision{Comparison of the methods for the second problem, the Eriksson-Johnson model problem. 
$L_2=100\frac{\|u-u_{exact}\|_{L_2}}{\|u_{exact}\|_{L_2}}$,  $H_1=100\frac{\|u-u_{exact}\|_{H_1}}{\|u_{exact}\|_{H_1}}$}. }
\label{tab:comparison_second}
\end{table}
\end{center}

\revision{In this section, we compare all the methods on the uniform grids. Namely, we compare the standard Galerkin formulation, Petrov-Galerkin formulation with the optimal test functions, Galerkin/least-squares method, and the SUPG method. The comparison for the first problem is summarized in Table \ref{tab:comparison_first}. We can read that the Galerkin/least-squares method provides the best solutions in the $L_2$ norm for small values of $\epsilon$. 
 The comparison for the second problem is summarized in Table \ref{tab:comparison_second}. It also shows, that the Galerkin/least-squares method provides the best solutions in the $L_2$ norm for small values of $\epsilon$. The estimates in the $H_1$ norm, show that Petrov-Galerkin method or SUPG method can deliver slightly better accuracy solutions for both problems. Still, the quality of all the results is similar. }

\section{Stability Analysis}

To show the stability of our Galerkin/least\revision{-}squares method, we consider the two-dimensional advection-diffusion equation:
\begin{equation*}
\label{eq:exact-formulation}
\begin{aligned}
  -\varepsilon\Delta u + \beta \cdot \nabla u &= f \\
  \left.u\right|_{\partial\Omega} &= 0
\end{aligned}
\end{equation*}
where~$\varepsilon > 0$ is a constant diffusion coefficient, $\beta$ is constant advection vector
and $f$ is a prescribed source.
For simplicity, let us consider~$\Omega = (0, L)^2$.
We consider the following discretization using B-splines on a uniform grid with element size~$h$ \revision{(defined as an element diameter)}:
\begin{equation}
\label{eq:stabilized-formulation}
  b(u_h, v_h) = l(v_h) \quad \forall v_h \in U_h,
\end{equation}
where:
\begin{equation}
\begin{aligned}
  b(u_h, v_h) &=
    \frac{1}{h}\left( \varepsilon\Prod{\nabla u_h}{\nabla v_h}_\Omega + \Prod{\beta \cdot \nabla u_h}{v_h}_\Omega \right) \\&+
    \Prod{-\varepsilon \Delta u_h + \beta \cdot \nabla u_h}
    {-\varepsilon \Delta v_h + \beta \cdot \nabla v_h}_\Omega,
\end{aligned}
\end{equation}
and:
\begin{equation}
  l(v_h) = (f, v_h)_\Omega.
\end{equation}

\revision{The goal of this section is to prove that the above discrete formulation
is well-posed.
To this end, we will demonstrate that the bilinear form~$b$ is continuous
and coercive on~$U_h$ equipped with the scalar product
\begin{equation}
    \Prod{u_h}{v_h}_{U_h} = 
    \frac{\varepsilon}{h}\Prod{\nabla u_h}{\nabla v_h}_{L_2} + 
    \Prod{\beta \cdot \nabla u_h}{\beta \cdot \nabla v_h},
    \label{eq:scalara}
\end{equation}
with constants independent of the mesh size~$h$.}

\revision{To do this, first we need an inverse inequality from the following lemma.}
\begin{lemma}
\label{lem:inverse-ineq}
There exists a constant~$C>0$ independent of~$h$ such that:
\begin{equation}
  \Norm{\Delta v_h}_{L_2} \leq C h^{-1} \Norm{\nabla v_h}_{L_2} \quad \forall v_h \in U_h,
\end{equation}
\begin{proof}
   For B-splines of order~$p$ on~$(0, 1)^d$, we have:
   \begin{equation}
      \left|v_h\right|_{H^1} \leq 2\sqrt{3d}p^2 h^{-1}\Norm{v_h}_{L_2},
   \end{equation}
   (see \cite{splines}).
   Scaling to~$(0, L)^d$ does not change the above inequality, since
   for~$v_h(x) = \revision{\tilde{v}}_h(x/L)$, where~$\tilde{v}_h$ is a B-spline function
   defined on~$(0, 1)^d$ we have
   \begin{equation}
       \left|v_h\right|_{H^1} = L^{d/2-1}\left|\tilde{v}_h\right|_{H^1},
       \quad
       \Norm{v_h}_{L_2} = L^{d/2}\Norm{\tilde{v}_h}_{L_2},
       \quad
       h = \tilde{h} L,
   \end{equation}
   where~$\tilde{h}$ is the size of the element of a mesh scaled to~$(0, 1)$,
   and thus
   \begin{equation}
   \begin{aligned}
      \left|v_h\right|_{H^1} &= L^{d/2-1}\left|\tilde{v}_h\right|_{H^1}
      \\&\leq 
      L^{d/2-1} 2 \sqrt{3d}p^2 \tilde{h}^{-1} \Norm{\tilde{v}_h}_{L_2}
      \\&=
      2 \sqrt{3d}p^2 (\tilde{h}L)^{-1} L^{d/2}\Norm{\tilde{v}_h}_{L_2}
      \\&= 2\sqrt{3d}p^2 h^{-1}\Norm{v_h}_{L_2}
   \end{aligned}
   \end{equation}
   Consequently,
   \begin{equation}
   \begin{aligned}
      \Norm{\Delta v_h}_{L_2}^2 &\leq \left|\nabla v_h\right|_{H^1}^2 =
      \sum_{k=1}^d \left|\partial_k v_h\right|_{H^1}^2
      \\&\leq
      \sum_{k=1}^d \left(2\sqrt{3d}(p-1)^2 h^{-1}\right)^2\Norm{\partial_k v_h}_{L_2}^2
      \\&\leq
      \left(2\sqrt{3d}(p-1)^2 h^{-1}\right)^2 \Norm{\nabla v_h}_{L_2}^2
   \end{aligned}
   \end{equation}
   since partial derivatives of~$v_h$ are B-splines of order~$p-1$.
   Therefore, a constant with value $C = 2\sqrt{3d}(p-1)^2$ satisfies the statement of the lemma.
\end{proof}
\end{lemma}


\revision{Using this results, we can establish continuity of~$b$,
provided sufficiently small~$\varepsilon > 0$.}

\revision{
\begin{lemma}
\label{lem:continuity}
Assuming~$\varepsilon \leq \frac{1}{2}h/C^2$,
where~$C$ is the constant from Lemma~\ref{lem:inverse-ineq},
\begin{equation}
  b(u_h, v_h) \leq M \Norm{u_h}_{U_h} \Norm{v_h}_{U_h}
\end{equation}
for some constant~$M > 0$ independent of~$h$.
\begin{proof}
Expanding the definition of~$b$, we have
\begin{equation}
\begin{aligned}
  b(u_h, v_h) &=
  \frac{\varepsilon}{h}\Prod{\nabla u_h}{\nabla v_h}
  \\&+ \Prod{-\varepsilon \Delta u_h + \beta \cdot \nabla u_h}{- \varepsilon \Delta v_h + \beta \cdot \nabla v_h}
  \\&+ \frac{1}{h}\Prod{\beta \cdot \nabla u_h}{v_h}
\end{aligned}
\end{equation}
The first term can be bounded as
\begin{equation}
   \frac{\varepsilon}{h}\Prod{\nabla u_h}{\nabla v_h}
   \leq
   \frac{\varepsilon}{h}\Norm{\nabla u_h}\Norm{\nabla v_h}
   \leq 
   \Norm{u_h}_{U_h} \Norm{v_h}_{U_h}.
\end{equation}
For the middle term, note that by Lemma~\ref{lem:inverse-ineq}:
  \begin{equation}
    \Norm{\varepsilon\Delta u_h}_{L_2} 
    =
    \varepsilon\Norm{\Delta u_h}_{L_2} 
    \leq \varepsilon Ch^{-1}\Norm{\nabla u_h}_{L_2},
  \end{equation}
and using the assumption on~$C$ we get
  \begin{equation}
  \label{eq:laplacian-inequality}
    \Norm{\varepsilon\Delta u_h}_{L_2}^2 \leq
    \frac{\varepsilon}{h} \left(\varepsilon C^2 h^{-1}\right)\Norm{\nabla u_h}_{L_2}^2
    \leq \frac{1}{2} \frac{\varepsilon}{h}\Norm{\nabla u_h}_{L_2}^2,
  \end{equation}
%
so that
\begin{equation*}
\begin{aligned}
  \frac{1}{2}\Norm{-\varepsilon\Delta u_h + \beta \cdot \nabla u_h}_{L^2}^2 
  &\leq
  \Norm{\varepsilon\Delta u_h}_{L^2}^2 + \Norm{\beta \cdot \nabla u_h}_{L^2}^2
  \\&\leq
  \frac{1}{2} \frac{\varepsilon}{h}\Norm{\nabla u_h}_{L^2}^2
  + \Norm{\beta \cdot \nabla u_h}_{L^2}^2
  \\&\leq
  \Norm{u_h}_{U_h}^2,
\end{aligned}
\end{equation*}
which shows that
\begin{equation*}
    \Prod{-\varepsilon \Delta u_h + \beta \cdot \nabla u_h}
    {- \varepsilon \Delta v_h + \beta \cdot \nabla v_h} 
    \leq 2 \Norm{u_h}_{U_h} \Norm{v_h}_{U_h}
\end{equation*}
Finally, to bound the last term, by Poincar\'e inequality there
exists a constant~$K > 0$ (depending only on~$\Omega$), such that
\begin{equation*}
  \Norm{u_h} \leq K \Norm{\nabla u_h}
\end{equation*}
for all~$u_h \in U_h$.
Therefore,
\begin{equation*}
\begin{aligned}
  \frac{1}{h}\Prod{\beta \cdot \nabla u_h}{v_h}
  &\leq
  \Par{h^{-1/2}\Norm{\beta \cdot \nabla u_h}}\,
  \Par{h^{-1/2}\Norm{v_h}_{L^2}}
  \\&\leq
  \varepsilon^{-1}
  \Par{
    |\beta| \varepsilon^{1/2}h^{-1/2}\Norm{\nabla u_h}_{L^2}
  }
  \Par{
  \varepsilon^{1/2}h^{-1/2} K \Norm{\nabla v_h}_{L^2}
  }
  \\&\leq
  \frac{K|\beta|}{\varepsilon}
  \Norm{u_h}_{U_h} \Norm{v_h}_{U_h}.
\end{aligned}
\end{equation*}
\end{proof}
\end{lemma}
}

\begin{theorem}
\label{thm:bounded-below}
    Assuming~$\varepsilon \leq \frac{1}{2}h/C^2$,
    where~$C$ is the constant from Lemma~\ref{lem:inverse-ineq},
    we have:
    \begin{equation}
      b(v_h, v_h) \geq \frac{1}{2}\left(
        \frac{\varepsilon}{h}\Norm{\nabla v_h}_{L_2}^2 + 
        \Norm{\beta \cdot \nabla v_h}_{L_2}^2
      \right) \quad \forall v_h \in U_h,
    \end{equation}
\end{theorem}
\begin{proof}
  We have:
  \begin{equation}
  \label{eq:b-vh-vh}
  \begin{aligned}
    b(v_h, v_h) &= \frac{1}{h}\left(
      \varepsilon\Norm{\nabla v_h}_{L_2}^2 + \Prod{\beta \cdot \nabla v_h}{v_h}
    \right)
    \\&+ \Norm{\varepsilon\Delta v_h}_{L_2}^2 + \Norm{\beta\cdot \nabla v_h}_{L_2}^2
    - 2\Prod{\varepsilon \Delta v_h}{\beta\cdot \nabla v_h}
  \end{aligned}
  \end{equation}
  The term~$\Prod{\beta \cdot \nabla v_h}{v_h}$ vanishes, since:
  \begin{equation}
  \begin{aligned}
    \Prod{\beta \cdot \nabla v_h}{v_h} &= \int_{\Omega}v_h(\beta\cdot \nabla v_h)\,dx
    = \int_{\Omega}v_h \nabla\cdot(\beta v_h)\,dx \\&=
    -\int_{\Omega}\nabla {v_h} \cdot (\beta {v_h})\,dx + \int_{\partial \Omega} v_{H_1} \beta \cdot dn
    \\&=-\Prod{\beta \cdot \nabla v_h}{v_h}
  \end{aligned}
  \end{equation}
  as~$v_h = 0$ on~$\partial\Omega$, i.e., a skew-symmetry condition.
  Furthermore, for any~$\alpha > 0$ we have:
  \begin{equation}
  \begin{aligned}
    2\left|\Prod{\varepsilon\Delta v_h}{\beta\cdot \nabla v_h}\right|
    &\leq
    2 \Norm{\varepsilon\Delta v_h}_{L_2} \Norm{\beta \cdot \nabla v_h}_{L_2}
    \\&\leq
    \alpha \Norm{\varepsilon\Delta v_h}_{L_2}^2 +
    \frac{1}{\alpha} \Norm{\beta \cdot \nabla v_h}_{L_2}^2
  \end{aligned}
  \end{equation}
  using Cauchy-Schwartz and Young's inequalities.
  In particular, for~$\alpha = 2$ we have:
  \begin{equation}
  2\left|\Prod{\varepsilon\Delta v_h}{\beta\cdot \nabla v_h}\right|
  \leq
  2 \Norm{\varepsilon\Delta v_h}_{L_2}^2 +
  \frac{1}{2} \Norm{\beta \cdot \nabla v_h}_{L_2}^2
  \end{equation}
  Combining this with equation~\eqref{eq:b-vh-vh}, we get:
  \begin{equation}
     b(v_h, v_h) \geq \frac{\varepsilon}{h}\Norm{\nabla v_h}_{L_2}^2
     - \Norm{\varepsilon\Delta v_h}_{L_2}^2 + \frac{1}{2}\Norm{\beta\cdot \nabla v_h}_{L_2}^2.
  \end{equation}
  %
  %
 \revision{Using the inequality~\eqref{eq:laplacian-inequality}
  \begin{equation}
    \Norm{\varepsilon\Delta v_h}_{L_2}^2
    \leq \frac{1}{2} \frac{\varepsilon}{h}\Norm{\nabla v_h}_{L_2}^2,
  \end{equation}
  established as part of the proof of continuity,
  }
  we arrive at the desired result:
  \begin{equation}
    b(v_h, v_h) \geq \frac{1}{2}\left(
      \frac{\varepsilon}{h}\Norm{\nabla v_h}_{L_2}^2 + 
      \Norm{\beta \cdot \nabla v_h}_{L_2}^2
    \right)
  \end{equation}
\end{proof}

We omit proofs for continuity of the bilinear and linear forms as this trivially follows using the Cauchy-Schwarz inequality. 
\revision{Together, all these properties guarantee stability of the formulation~\cite{babuska-inf-sup}.}

\section{Conclusions}

In this paper, we have presented an isogeometric analysis based least-squares stabilization in which we use an exact formula for the optimal test functions. This follows from the selection of the $L_2$ norm in the minimization of the residual. For this setup we present numerical experiments showing the stability properties and note that the resulting solutions are overly diffused in the case of coarse uniform meshes. To alleviate this issue, we derived Galerkin/least squares formulation, introducing combined test functions.
For the Galerkin/least squares we performed theoretical analysis. 

\revision{The motivation of our stabilization method is to modify the optimal test functions employed by the Petrov-Galerkin formulation with the optimal test functions, following the first term in the proof of the inf-sub stability in equation (\ref{eq:b-vh-vh}).
This first term allows us to obtain the inf-sub stability in the weighted graph norm defined by (\ref{eq:scalara}).
For the optimal test functions without our extra term, the inf-sub stability can be proven only in the $L_2$ norm, so the convergence of the error happens only in the $L_2$ norm, which does not see the oscillations of the solution.
Thus, the intuition behind this definition is to make it possible to prove the inf-sub stability in a better norm, the weighted graph norm.}

We compare the numerical results for 
\revision{two model advection-diffusion problems}
solved using ten elements \revision{on} a uniform mesh. The obtained solutions with optimal test functions computed from residual minimization in $L_2$ norm, presented in Table\revision{s~\ref{tab:tab3A} and~\ref{tab:tab3}}, with the modified stabilization using modified test functions, shown in Table\revision{s~\ref{tab:tab4A} and~\ref{tab:tab4}}, and finally with the SUPG method, in Table\revision{s~\ref{tab:tab5A} and~\ref{tab:tab5}}.
The best quality of the numerical results is obtained from the modified Galerkin / least\revision{-}squares formulation.
The optimal test functions computed with $L_2$ norm generally behave better than SUPG stabilization for $\epsilon=0.1$ or $\epsilon=0.01$. 
For small $\epsilon=0.001$ or $\epsilon=0.0001$, the SUPG and the modified stabilization deliver similar quality results, and the optimal test functions computed with $L_2$ norm have the largest errors. We attribute this to its overly diffusive nature for small $\epsilon$.
The Galerkin/least\revision{-}squares combined stabilization method  
delivers similar and sometimes better quality results than the SUPG method.  It may be an attractive alternative to stabilize numerical simulations as the proposed method requires no tuning of stabilization parameters thereby easing implementation for new cases. 

\section{Acknowledgements}

Research project supported by the program ``Excellence initiative - research university" for the AGH University of Science and Technology.

\end{document}